\documentclass[psamsfonts]{amsart}
\usepackage{amsmath,amssymb}
\usepackage{enumerate}
\newtheorem{theorem}{Theorem}
\newtheorem{corollary}[theorem]{Corollary}
\newtheorem{lemma}[theorem]{Lemma}
\newtheorem{proposition}[theorem]{Proposition}
\theoremstyle{definition}

\theoremstyle{remark}
\newtheorem{remark}[theorem]{Remark}
\theoremstyle{remark}
\newtheorem{example}[theorem]{Example}

\def\im{\operatorname{Im}}
\def\aut{\operatorname{Aut}}
\def\prop{\operatorname{Prop}}
\def\rea{\operatorname{Re}}

\title[Proper mappings]{Geometry of quasi-circular domains and applications to tetrablock}
\author[\L.~Kosi\'nski]{\L ukasz Kosi\'nski}
\address{Instytut Matematyki\\ Uniwersytet Jagiello\'nski\\ \L ojasiewicza 6,
30-348 Krak\'ow\\ Poland} \email{lukasz.kosinski@gazeta.pl} 
\thanks{The paper was partially supported by the Research Grant of the Polish Ministry of Science and Higher Education N$^{\text o}$ \textbf{N N201 271435}}

\begin{document}
\maketitle

\begin{abstract}We prove that the Shilov boundary is invariant under proper holomorphic mappings between some classes of domains (containing among others quasi-balanced domains with the continuous Minkowski functionals). Moreover, we obtain an extension theorem for proper holomorphic mappings between quasi-circular domains. 

Using these results we show that there are no non-trivial proper holomorphic self-mappings in the tetrablock. Another important result of our work is a description of Shilov boundaries of a large class of domains (containing among other the symmetrized polydisc and the tetrablock). 

It is also shown that the tetrablock is not $\mathbb C$-convex.
\end{abstract}

\section{Introduction and statement of results}
In the paper we will use the notion of \emph{quasi-circular} domains. Let $m_1,\ldots,m_n$ be relatively prime natural numbers. Recall that a domain $D\subset\mathbb C^n$ is said to be $(m_1,\ldots,m_n)$-circular (shortly \textit{quasi-circular}) if \begin{equation}\label{defcir} (\lambda^{m_1}x_1,\ldots,\lambda^{m_n}x_n) \in D\quad \text{for any}\quad |\lambda|=1,\ x=(x_1,\ldots,x_n)\in D.\end{equation} If the relation (\ref{defcir}) holds with $|\lambda|\leq 1,$ then $D$ is said to be $(m_1,\ldots,m_n)$-balanced (shortly \textit{quasi-balanced}).

Let $\mathcal R_{II}$ denote the classical Cartan domain of the second type, i.e. $$\mathcal R_{II}=\{z\in\mathcal{M}_{2\times2}(\mathbb C):\ z=z^t,\ ||z||<1\},$$ where $||\cdot||$ is the operator norm and $\mathcal M_{2\times 2}(\mathbb C)$ denotes the space of $2\times 2$ complex matrices. Put $$\Pi:\mathcal{M}_{2\times2}(\mathbb C)\ni z=(z_{i,j})\rightarrow (z_{1,1},z_{2,2},\det z)\in\mathbb C^3.$$ We define $\mathbb E:=\Pi(\mathcal R_{II}).$ The domain $\mathbb E$ is called the \emph{tetrablock}.

The tetrablock is a $(1,1,2)$-balanced domain in $\mathbb C^3$ appearing in control engineering and produces problems of a function-theoretic character. Its geometric properties have been investigated in several papers (see e.g. \cite{You0}, \cite{Zwo}, \cite{You}  and references contained there). Recall here that in \cite{You} the author using Kaup's theorem obtained a description of the group of automorphisms of this domain. In the paper we prove an Alexander-type theorem for the tetrablock showing that every proper holomorphic self-map of the tetrablock is an automorphism.

\begin{theorem}\label{main}
 Let $\varphi:\mathbb E\to\mathbb E$ be a proper holomorphic mapping. Then $\varphi$ is an automorphism.
\end{theorem}
As a side effect we obtain a natural correspondence between automorphisms of the tetrablock and of the classical domain of the second type indicated in Lemma~\ref{lemma2}. This correspondence gives much easier and more elementary method of deriving the explicit formulas for automorphisms of the tetrablock. In particular, we extend results from \cite{You} and simplify their proofs.

The methods used in the paper rely upon the investigation of proper holomorphic mappings between quasi-circular domains. We start with generalizing the Bell's extension result (see \cite{Bel}). Next we analyze the behavior of the Shilov boundary under proper holomorphic mappings. We have the following
\begin{theorem}\label{shpr}Let $D$ and $G$ be bounded domains in $\mathbb C^n$ and let $f:D\to G$ be a proper holomorphic mapping extending continously to $\overline D.$ Assume that there is an increasing family of domains $\{G_m\},$ $G_m\Subset G_{m+1},$ such that $\bigcup G_m=G$ and $\overline{(\bigcup_m \partial_s G_m)}\cap\partial G=\partial_s G.$

Then $f(\partial_s D)=\partial_s G.$
\end{theorem}
Note that in general the Shilov boundary is not invariant even under biholomorphic polynomial mappings - see Example~\ref{example}.

Based on the former idea we also obtain the following result:
\begin{theorem}\label{lempr}Let $f:D\to G$ be a proper holomorphic mappings between domains in $\mathbb C^n.$ Let $L$ be a domain relatively compact in $G.$ Put $K=f^{-1}(L).$ Then \begin{equation}f(\partial_s K)=\partial_s L\quad \text{and}\quad f(\partial_b K)=\partial_b L.\end{equation}
\end{theorem}

As a consequence of our considerations we show that any proper holomorphic mapping between quasi-balanced bounded domains preserves the Shilov boundary. Namely, we have the following
\begin{corollary}\label{cor} \textit{a)} Let $D\subset\mathbb C^n$ be a bounded domain and let $G$ be a bounded quasi-balanced domain in $\mathbb C^n.$ Assume that the Minkowski functional associated to $G$ is continuous and for any open, relatively compact subset $K$ of $D$ there is an open neighborhood $U$ of $\overline D$ such that $k_D(z,\overline w)$ extends holomorphically on $U\times \tilde K.$

Then every proper holomorphic mapping $f:D\to G$ maps $\partial_s D$ onto $\partial_s G.$

b) Let $D$ and $G$ be bounded quasi-balanced domains. If the Minkowski functionals of $D$ and $G$ are continuous, then any proper holomorphic mapping between $D$ and $G$ preserves the Shilov boundary.
\end{corollary}

As we indicate in the sequel the results obtained in the paper give immediately a description of the Shilov boundaries of many domains like the symmetrized polydisc (see \cite{Edi-Zwo}), the tetrablock (see \cite{You0}) etc. Moreover, they exclude the existence of proper holomorphic mappings between some domains. For example, the well known theorem stating that there is no proper holomorphic mapping between the polydisc and the Euclidean ball is a direct consequence of our results. 

\bigskip

As a by-product of our considerations we obtain in Lemma~\ref{bih}  an extension of the main result from \cite{Tum-Hen}.

\bigskip

In this paper in Remark~\ref{rem} it is also shown that the tetrablock is not $\mathbb C$-convex. Recall that a consequence of the Lempert theorem is the fact that the Carath\'eodory pseudodistance and the Lempert function of a $\mathbb C$-convex domain with $\mathcal C^2$ boundary coincide (see \cite{Jac}). Since results obtained in \cite{You0} (see also \cite{Zwo}) suggest that the equality between the Carath\'eodory pseudodistance and the Lempert function holds in the tetrablock, the tetrablock is the candidate for the first bounded pseudoconvex domain non-biholomorphically equivalent to a $\mathbb C$-convex domain for which the equality between mentioned above functions holds.

It also seems to be interesting whether the tetrablock may be exhausted by domains biholomorphic to $\mathbb C$-convex domains.

\bigskip

Here is some notation. Throughout the paper $\mathbb D$ denotes the unit disc in the complex plane. The unit Euclidean ball in $\mathbb C^n$ is denoted by $\mathbb B_n.$ Moreover $\prop(D,G)$ is the set of proper holomorphic mappings between domains $D$ and $G$. The Shilov and Bergman boundary is denoted respectively by $\partial_s$ and $\partial_b.$

\bigskip 

Now I would like to thank professor W\l odzimierz Zwonek for reading the manuscript, many remarks and fruitful discussions.  

\section{Extension of proper holomorphic mappings between quasi-balanced domains}

We start this section with recalling basic properties of circular domains and the Bergman projection which will be useful in the sequel. By $k_D$ we shall denote the Bergman kernel associated to a domain $D.$ Let moreover $P_D$  denote the Bergman projection for $D.$ We use the notation $$k_D^{\alpha}(z,w)=\partial^{\alpha}k_D(z,w)\quad \text{and}\quad k_D^{\overline{\alpha}}(z,w)=\partial^{\overline{\alpha}}k_D(z,w),$$ where $\partial^{\alpha}$ stands for $\frac{\partial^{|\alpha|
}}{z^{\alpha}}$ and $\partial^{\overline{\alpha}}$ stands for $\frac{\partial^{|\alpha|}}{\overline w^{\alpha}}.$

For a given $(m_1,\ldots,m_n)$-balanced domain $D$ in $\mathbb C^n,$ where $m_1,\ldots,m_n$ are relatively prime natural numbers, we define the \emph{Minkowski functional}
\begin{equation}\mu_D(x):=\inf\{\lambda>0:\ (\lambda^{-m_1}x_1,\ldots,\lambda^{-m_n}x_n)\in D\},\quad x=(x_1,\ldots,x_m)\in\mathbb C^n.
\end{equation}
This function has similar properties as the standard Minkowski functional for balanced domains. Some of them may be found in \cite{Nik}. In particular,
\begin{align} &\mu_D(\alpha^{m_1}x_1,\ldots,\alpha^{m_n}x_n)=|\alpha|\mu_D(x),\quad x\in\mathbb C^n,\ \alpha\in\mathbb C,\\
&D=\{x\in\mathbb C^n:\ \mu_D(x)<1\}.
\end{align}

For a subset $K$ of $\mathbb C^n$ we put $\tilde K:=\{\overline x:\ x\in K\}.$

\begin{remark}\label{poznauwaga}Let $D$ be an $(m_1,\ldots,m_n)$-balanced bounded domain whose Minkowski functional is continuous. Put $D_r:=\{x\in\mathbb C^n:\ \mu_D(x)<r\},$ $r>0.$ Since $$k_D((r^{m_1}z_1,\ldots, r^{m_n} z_n),w)=k_D(z,(r^{m_1}w_1,\ldots, r^{m_n}w_n))$$ for $z,w\in D,$ $r\in [0,1],$ we easily find that the function $(z,w)\to k_D(z,\overline w)$ may be extended holomorphically to $D_{1/r}\times \tilde{D_r}$ for any $0<r\leq 1.$ 
\end{remark}

It follows from \cite{Bel0} that if $f:D\to G$ is a proper holomorphic mapping between bounded domains $D,$ $G$ in $\mathbb C^n,$ then for any $\Phi\in L^2(G)$ we have \begin{equation}\label{row}P_{D}(\det[f']\cdot(\Phi\circ f))=\det[f'] \cdot ((P_G\Phi)\circ f).\end{equation} 

Assume additionally that $G$ is an $(m_1,\ldots,m_n)$-circular domain containing the origin. Choose $\delta>0$ such that $\delta\overline{\mathbb B_n}\subset G.$ Let $\theta$ be a radial function in $\mathcal{C}^{\infty}_0(\delta\mathbb B_n)$ such that $\theta\geq 0$ and $\int_{\delta\mathbb B_n}\theta=1$. Since holomorphic functions assume their average values we find that
\begin{equation}\label{3}\partial^{\alpha}h(0) = \int_{G}(\partial^{\alpha}h)\theta d\lambda^{2n}=\int_{G} h(-1)^{|\alpha|} \partial^{\alpha}\theta d\lambda^{2n}\end{equation}  for every $h\in\mathcal O(G)\cap L^2(G).$ 

On the other hand $h(z)=\int_{G}k_{G}(z,w)h(w) d\lambda^{2n}(w),$ $z\in G.$ Since $k_G(z,\cdot)$ extends holomorphically to a neighborhood of $\overline G$ provided that $z$ is sufficiently close to $0$, one may differentiate this formula at $z=0$ to get that \begin{equation}\label{ker0}\partial^{\alpha}h(0)=\int_{G}\partial^{\alpha}k_{G}(0,w)h(w) d\lambda^{2n}(w),\quad h\in\mathcal O(G) \cap L^2(G).\end{equation}
This relation together with (\ref{3}) gives \begin{equation}\label{ker}P_G((-1)^{|\alpha|}\overline{\partial}^{\alpha}\theta)=k_G^{\overline{\alpha}}(\cdot,0). \end{equation}

The next lemma has been proved by S.~Bell in the case when $D$ and $G$ are bounded circular domains and $0\in G$ (see \cite{Bel}). It is interesting that after minor modifications the methods used by Bell yield a stronger result. We present the whole proof for the sake of completeness.

\begin{lemma}\label{text}Let $D,G$ be bounded domains in $\mathbb C^n.$ Suppose that $G$ is $(m_1,\ldots,m_n)$-circular and contains the origin. Assume moreover that the domain $D$ satisfies the following property: for any open, relatively compact subset $K$ of $D$ there is an open set $U$ containing $\overline D$ such that $(z,w)\to k_D(z,\overline w)$ extends holomorphically to $U\times \tilde K.$

Then any proper holomorphic mapping $f:D\to G$ extends holomorphically to a neighborhood of $\overline{D}.$
\end{lemma}

\begin{proof}
Let $m=(m_1,\ldots,m_n).$ Properties of the Bergman kernel and a standard argument imply that the equation \begin{equation}k_G((\lambda^{m_1}z_1,\ldots, \lambda^{m_n} z_n),w)=k_G(z,(\overline{\lambda^{m_1}}w_1,\ldots, \overline{\lambda^{m_n}}w_n))\end{equation} holds for any $z,w\in G$ and $|\lambda|$ sufficiently close to $1.$
Differentiating this formula several times with respect to $\overline w_i$ and putting $w=0$ we find that \begin{equation}\frac{\partial^{\alpha} k_G}{\partial\overline w^{\alpha}}((\lambda^{m_1}z_1,\ldots, \lambda^{m_n} z_n),0)=\lambda^{\langle \alpha,m\rangle}\frac{\partial^{\alpha} k_G}{\partial\overline w^{\alpha}}(z,0)\end{equation} for $\alpha\in\mathbb N^n,$ $z\in G$ and $|\lambda|$ sufficiently close to $1.$

Whence a standard argument shows that there are $c_{\beta}\in\mathbb C$ such that \begin{equation}\label{z1}k_G^{\overline{\alpha}} (z,0)= \sum c_{\beta}z^{\beta},\ z\in G,\end{equation} where the sum is taken over $\beta\in\mathbb N^n$ satisfying the relation $\langle\beta, m \rangle= \langle \alpha, m \rangle.$ Therefore, the linear independence of $k_G^{\overline{\alpha}}(z,0)$ (see (\ref{ker0})) implies that for every $\beta\in\mathbb N^n$ there are $\tilde c_{\alpha}$ such that \begin{equation}\label{gen}z^{\beta}=\sum \tilde c_{\alpha} k_G^{\overline{\alpha}} (z,0),\end{equation} where the sum is taken over $\alpha\in\mathbb N^n$ satisfying the relation $\langle\alpha, m \rangle= \langle \beta, m \rangle.$

Now (\ref{gen}) together with (\ref{ker}) provide us with the function $\phi_{i,k}\in\mathcal C_0^{\infty}(\delta\mathbb B_n)$ such that \begin{equation}z_i^k=P_G(\Phi_{i,k}),\quad i=1,\ldots,n,\ k\in\mathbb N.\end{equation} Making use of the above relations we infer that
\begin{equation}\label{ext}\det[f'(z)]f_i^k(z)=\det[f'(z)](z_i^k\circ f(z))=\int_D k_D(z,w)\det[f'(w)]\Phi_{i,k}(f(w)) d\lambda^{2n} (w),\end{equation} for $i=1,\ldots,n,$ and $k\in\mathbb N.$ From these relations and the assumption on $k_D$ we easily conclude that all the functions appearing in the left side of (\ref{ext}) extend holomorphically to some open, connected neighborhood $U$ of $\overline D.$

We will briefly show that $f_i$ extends holomorphically to the domain $U.$ Putting $u=\det[f']$ we have the following situation $$u\in\mathcal O(U),\quad u\not\equiv 0 \quad\text{and}\quad uf_i^k\in\mathcal O(U),\ k\in\mathbb N.$$ Fix any point $x\in U$ such that $u(x)=0.$ Changing, if necessary, the coordinates system we may assume that both $u$ and $uf_i$ satisfy the assumptions of Weierstrass Preparation Theorem near $x.$ Since $uf_i^k$ is holomorphic on $U$, the Weierstrass polynomial associated to $u$ divides the Weierstrass polynomial associated to $uf_i.$ This, in particular, means that $f_i$ is locally bounded near the analytic set $\{u=0\},$ so the assertion follows from the Riemann's removable singularity theorem.
\end{proof}

\begin{remark}
Note that the continuity of the Minkowski functional of a bounded quasi-balanced domain $D$ is equivalent to the fact that for every $0<r<1$ the domain $D$ is relatively compact in $D_{1/r}.$ Therefore any quasi-balanced domain fulfils the assumptions of Lemma~\ref{text}.
\end{remark}

\begin{corollary}\label{pext}Any proper holomorphic mapping $f:\mathcal R_{II}\to\mathbb E$ may be extended holomorphically to a neighborhood of $\overline{\mathcal R_{II}}.$
\end{corollary}

\section{Proofs of Theorems~\ref{shpr}, \ref{lempr} and the applications}
We start this section with the following
\begin{remark}The technical assumption occurring in the Theorem~\ref{shpr} seems to be very natural. Observe that $x\in \overline{(\bigcup_m \partial_s G_m)}\cap\partial G$ if and only if there is a subsequence $(n_k)$ and there are $x_{n_k}\in \partial_s G_{n_k}$ such that $x_{n_k}\to x.$ 

One may very easily show that for any bounded domain $G$ and any increasing family of domains $\{G_m\}$ such that $\bigcup G_m=G,$ the Shilov boundary of $G$ is contained in $\overline{(\bigcup_m \partial_s G_m)}\cap\partial G.$\end{remark}

\begin{example}\label{example}
Note  that Theorem~\ref{shpr} does not remain valid if we remove the assumption $\overline{(\bigcup_m \partial_s G_m)}\cap\partial G=\partial_s G$ even in the case when $f$ is a proper polynomial mapping. As an example one may take $D=\mathbb D\cap\{z\in\mathbb D:\ \im z>0\},$ $G=\mathbb D\setminus [0,1)$ and $f(z)=z^2.$
\end{example}

\begin{proof}[Proof of Theorem~\ref{shpr}] The inclusion $\partial_s G\subset f(\partial_s D)$ follows immediately from the definition of the Shilov boundary. We shall prove that $\partial_s D\subset f^{-1}(\partial_s G).$ Assume the contrary i.e. there is a $\psi\in\mathcal O(D)\cap \mathcal C(\overline D)$ such that \begin{equation}\label{our}|\psi(x_0)|>\max\{|\psi(x)|:\ x\in f^{-1}(\partial_s G)\},\end{equation} for some $x_0\in\partial D.$ Note that $$\limsup_{m\to\infty}\max\{|\psi(x)|:\ x\in D\cap f^{-1}(\partial_s G_m)\}\leq\max\{|\psi(x)|:\ x\in \partial D\cap f^{-1}(\partial_s G)\}.$$ Actually, otherwise there would exist a subsequence $(m_k)\subset\mathbb N,$ $\epsilon>0$ and $x_{m_k}\in D\cap f^{-1}(\partial_s G_{m_k})$ such that $$|\psi(x_{m_k})|>\max\{|\psi(x)|:\ x\in \partial D\cap f^{-1}(\partial_s G)\}+\epsilon.$$ Passing, if necessary, to a subsequence we can assume that $x_{m_k}$ converges to some $x_0.$ Using the assumptions on the domain $G$ and the mapping $f$ we infer that $f(x_0)\in\partial_s G.$ Thus $$|\psi(x_0)|\geq\max\{|\psi(x)|:\ x\in \partial D\cap f^{-1}(\partial_s G)\}+\epsilon\quad\text{and}\quad x_0\in\partial D\cap f^{-1} (\partial_s G),$$ which gives an obvious contradiction.

Therefore we may take $m$ big enough and replace $x_0$ by a point $x_0'\in f^{-1}(\overline G_m)$ sufficiently close to $x_0$ at which the mapping $f$ is non-degenerate so that \begin{equation}\label{aa}|\psi(x_0')|>A:=\max\{|\psi(x)|:\ x\in D\cap f^{-1}(\partial_s G_m)\},\quad \#f^{-1}(f(x_0'))=k,\end{equation} where $k$ denotes the multiplicity of $f.$ 

Let $h_j,\ j=1,\ldots, k,$ be holomorphic mappings in the neighborhood of $f(x_0')$ given by $f^{-1}=\{h_j:\ j=1,\ldots,k\}.$ Making use of (\ref{aa}) together with the Kronecker Theorem (see e.g. \cite{Har}) one may show the existence of a natural number $d$ such that 
\begin{equation}\label{aaa}|\psi(h_1(f(x_0')))^d+\ldots +\psi(h_k(f(x_0')))^d|>kA^d.\end{equation} To prove it put $a_j=\psi(h_j(f(x_0'))),$ $j=1,\ldots,k.$ Change, if necessary, the order of $a_j$ so that $|a_1|=\ldots=|a_l|$ and $|a_j|<|a_1|$ for $j=l+1, \ldots, n.$ Dividing all $a_j$ by $\psi(h_1(f(x_0')))$ we reduce ourselves to the following situation: $$a_j=e^{i\theta_j},\ j=1\ldots,l,\quad |a_j|<1,\ j=l+1,\ldots,n\quad \text{and}\quad A<1,$$ where $\theta_j\in \mathbb R,$ $j=1,\ldots,l.$ 

Changing the order once again we may assume that $1,\theta_1,\ldots ,\theta_{l_1}$ are $\mathbb Q$-linearly independent, $l_1\leq l,$ $l_1\in\mathbb N\cup \{0\}$ and $$\theta_j=\frac{q_{j,0}}{N}+\sum_{\iota=1}^{l_1} \frac{q_{j,\iota}}{N} \theta_{\iota},$$ $j=l_1+1,\ldots,l,$ where $q_{j,\iota}\in\mathbb Z,$ $\iota=0,\ldots,l_1,$ and $N\in\mathbb N.$  Put $M=\max\{|q_{j,\iota}|,N\}.$ 

According to the Kronecker Theorem (see e.g. \cite{Har}) there is a sequence of natural numbers $(\tilde d_{\mu})$ such that $-\frac{1}{(\mu+1) k M} < \arg (e^{2\pi i\tilde d_{\mu} \theta_j}) <\frac{1}{(\mu+1) k M},$ $j=1,\ldots,l_1,$ $\mu\in\mathbb N.$ In particular, $-\frac{1}{\mu+1}< \arg(e^{2\pi i d_{\mu}\theta_j})<\frac{1}{\mu+1}$ for $j=1,\ldots,l,$ $\mu\in\mathbb N,$ where $d_{\mu}:=N\tilde d_{\mu}.$

Properties of $(d_{\mu})$ guarantee that $|a_1^{d_{\mu}}+\ldots +a_l^{d_{\mu}}| \to l$ as $\mu\to \infty.$ Since $d_{\mu}\to\infty,$ we find that $kA^{d_{\mu}}\to 0$ and $|a_{l+1}^{d_{\mu}}+\ldots +a_k^{d_{\mu}}| \to 0,$ $\mu\to \infty.$ Therefore $|a_1^{d_{\mu}} +\ldots+a_k^{d_{\mu}}|-kA^{d_{\mu}}\to l>0,$ which obviously proves the existence of a natural number $d$ fulfilling (\ref{aaa}).

\medskip
Put \begin{equation}\zeta(x)=x_1^d+\ldots+ x_k^d,\quad \text{for}\quad x=(x_1,\ldots,x_k)\in\mathbb C^k.\end{equation} A well known argument shows that the formula $\varphi=\zeta\circ(\psi\times\ldots\times\psi)\circ f^{-1}$ defines a holomorphic function on $G.$ It follows from (\ref{aa}) and (\ref{aaa}) that $$|\varphi(f(x_0'))|>\max\{|\varphi(x)|:\ x\in \partial_s G_m\};$$ a contradiction.
\end{proof}

\begin{remark}\label{glupiauwaga}It is clear that the proof remains valid if the assumption $\overline{(\bigcup_m \partial_s G_m)}\cap\partial G=\partial_s G$ occurring in Theorem~\ref{shpr} is replaced by a weaker condition $\overline{(\bigcup_m \partial_b G_m)}\cap\partial G=\partial_s G,$ where $\partial_b$ denotes the Bergman boundary.\end{remark}

The following result is a direct consequence of Theorem~\ref{shpr}:
\begin{proposition}Let $D$ and $G$ be bounded domains in $\mathbb C^n$ and let $f:D\to G$ be a proper holomorphic mapping extending holomorphically  to a neighborhood of $\overline D.$ Assume that there is an increasing family of domains $\{G_m\},$ $G_m\Subset G_{m+1},$ such that $\bigcup G_m=G$ and $\overline{(\bigcup_m \partial_b G_m)}\cap\partial G=\partial_b G.$

Then $f(\partial_b D)=\partial_b G.$
\end{proposition}
\begin{proof}It follows from Remark~\ref{glupiauwaga} that $\partial_s G=\partial_b G.$ Thus, Theorem~\ref{shpr} together with Remark~\ref{glupiauwaga} gives: $$\partial_b D\subset \partial_s D\subset f^{-1}(\partial_s G)=f^{-1}(\partial_b G).$$ The inclusion $\partial_b G\subset f(\partial_b D)$ may be shown as in the proof of Theorem~\ref{shpr}.
\end{proof}

\begin{proof}[Sketch of proof of Theorem~\ref{lempr}] The inclusions $\partial_s L\subset f(\partial_s K)$ and $\partial_b L\subset f(\partial_b K)$ are clear.

It remains to show that $\partial_s K\subset f^{-1}(\partial_s L)$ and $\partial_b K\subset f^{-1}(\partial_b L).$ We will prove both inclusion simultaneously. Assume a contrary, i.e. there is a function $\psi\in \mathcal O(K)\cap \mathcal C(\overline G)$ (respectively $\psi\in \mathcal O(\overline K)$) such that $|\psi|$ does not attain its maximum on $f^{-1}(\partial_b L),$ i.e. $$|\psi(x_0)|>A:=\max\{ |\psi(x)|:\ x\in f^{-1}(\partial_s L)\}$$ (resp. $|\psi(x_0)|>A:=\max\{ |\psi(x)|:\ x\in f^{-1}(\partial_b L)\}$) for some $x_0\in K.$ Obviously $f(x_0)$ may be assumed to be a regular value of $f.$

Let $k$ denote the multiplicity of the mapping $f:D\to G.$ Clearly, $f|_K:K\to L$ is also of multiplicity $k.$ 

Write $f^{-1}=\{h_1, \ldots, h_k\}$ in a neighborhood of $a$, where $h_i$ are holomorphic functions. One may repeat the argument used in the proof of Theorem~\ref{shpr} to show the existence of a natural number $d$ such that \begin{equation}\label{aaa2}|\psi(h_1(f(x_0)))^d+\ldots +\psi(h_k(f(x_0)))^d|>kA^d.\end{equation} Define $\zeta(x)=x_1^d+\ldots+ x_k^d,$ for $x=(x_1,\ldots,x_k)\in\mathbb C^k.$ A function $\varphi$ given by the formula $\varphi=\zeta\circ(\psi\times\ldots\times\psi)\circ f^{-1}$ is holomorphic $L$ and continuous on $\overline{f^{-1}(L)}=f^{-1}(\overline L)$ (resp. $\varphi$ is holomorphic in an open neighborhood of $\overline L$). It follows from (\ref{aaa2}) that $$|\varphi(f(x_0))|>\max\{|\varphi(x)|:\ x\in \partial_s L\};$$ a contradiction.
\end{proof}

\begin{proof}[Proof of Corollary~\ref{cor}]\textit{a)} Define $$G_{m}:=\left\{x\in \mathbb C^n:\ \mu_G(x)<1-\frac{1}{m}\right\},\quad m=2,3\ldots.$$ It is clear that the family $\{G_m\}_m$ satisfies the assumptions of Theorem~\ref{shpr}. So applying Lemma~\ref{text} we reduce the situation to the one occurring in Theorem~\ref{shpr}.

\emph{b)} It is a direct consequence of \emph{a)} and Remark~\ref{poznauwaga}
\end{proof}

\begin{remark}Note that Theorem~\ref{lempr} and Corollary~\ref{cor} allow us to determine the Shilov boundary of some classes of domains containing the symmetrized polydisc (see \cite{Edi-Zwo}) and the tetrablock. For example $\partial_s\mathbb E=\Pi(\partial_s\mathcal R_{II})=\Pi(\mathcal U),$  where $\mathcal U$ consists of unitary symmetric matrices (see also~\cite{You}, where the author using elementary methods computed the Shilov boundary of the tetrablock).

It is also interesting that Theorem~\ref{shpr} may be used for showing the non-existence of proper holomorphic mappings between some domains. For instance, using Corollary~\ref{cor} we immediately see that $\prop(\mathbb D^n,\mathbb B_n)$ and $\prop(\mathbb B_n,\mathbb D^n)$ are empty for $n\geq 2$ (see also \cite{Nar}). As an other example of the application of this result, observe that the theorem showing that there are no proper holomorphic mappings between $\mathbb B_n\times\mathbb B_m$ and $\mathbb B_{n+m}$ follows directly from Corollary~\ref{cor}.
\end{remark}

\section{Applications to the tetrablock}
The next result has been proved in \cite{Rud} for the Euclidean ball in $\mathbb C^n.$ We would like to mention here that for our purposes a much weaker result of Tumanov and Henkin proved in \cite{Tum-Hen} is sufficient. However, it seems to be interesting that after some modifications the Rudin's idea may be applied to the symmetric domains.

First recall a well known classical result.
\begin{lemma}[see \cite{Rud1}, Theorem 8.1.2]\label{lemrud} Suppose that $\Omega_1$ and $\Omega_2$ are balanced domains in $\mathbb C^n$ and $\mathbb C^m$ respectively. Suppose moreover that $\Omega_2$ is convex and bounded and $F:\Omega_1\to\Omega_2$ is holomorphic. Then $F'(0)$ maps $\Omega_1$ into $\Omega_2.$

If additionally $F(0)=0,$ then $F(\lambda\Omega_1)\subset \lambda\Omega_2,$ $0\leq\lambda\leq1.$
\end{lemma}

\begin{lemma}\label{bih} Let $a_0,$ $b_0$ be any unitary symmetric matrices. Let $U,$ $V$ be open neighborhoods of $a_0$ and $b_0$ respectively. Let $\varphi:U\cap\mathcal R_{II}\to V\cap\mathcal R_{II}$ be a biholomorphic mapping. If $\varphi(a_k)\to b_0$ for some $a_k\to a_0,$ then $\varphi$ extends to an automorphism of $\mathcal R_{II}.$
\end{lemma}

\begin{proof}
A direct computation shows that for any symmetric unitary matrix $a$ there is a unitary matrix $u$ such that $uu^t=a.$ Since any of the mappings $\mathcal R_{II}\ni x\to uxu^t\in\mathcal R_{II},$ where $u$ is unitary, is an automorphism of $\mathcal R_{II},$ we may assume that $a_0=b_0=1.$ 

Recall that (see e.g. \cite{Hua}) for every $a\in\mathcal R_{II}$ the mapping \begin{equation}\label{wzornaaut}\varphi_a(x)=-a+(1-aa^*)^{1/2}x(1-a^*x)^{-1}(1-a^*a)^{1/2}\end{equation} is an automorphism of $\mathcal R_{II}$, and $\varphi_a(0)=-a$ and its inverse is given by $\varphi_a^{-1}=\varphi_{-a}$.

Put $b_k=\varphi(a_k)$ and define $G_k:=\varphi_{b_k}\circ\varphi\circ\varphi_{-a_k}:\varphi_{a_k}(U\cap \mathcal R_{II})\to \varphi_{b_k}(V\cap\mathcal R_{II}),\ k\in\mathbb N.$ Note that $G_k$ is a biholomorphic mapping, $G_k(0)=0.$ Clearly $\varphi_{-a}(x)\to 1$ locally uniformly whenever $a\to1,$  so a compactness argument gives the existence of $\delta_k>0$ such that $\delta_k\to 1,$ as $k\to\infty,$ and both $\varphi_{a_k}(U\cap \mathcal R_{II}),$ $\varphi_{b_k}(V\cap\mathcal R_{II})$ contain a domain $\delta_k\mathcal R_{II}.$ Properly scaled Lemma~\ref{lemrud} implies that $\delta_k^3\leq |\det G_k'(0)|\leq \delta_k^{-3}.$ 

Since $G_k(0)=0,$ it follows that there exists a subsequence of $\{G_k\}$ (also denote by $\{G_k\}$) converging locally uniformly to $G:\mathcal R_{II}\to\mathcal R_{II}.$ Clearly $|\det G'(0)|=1$ and $G(0)=0,$ so by Lemma~\ref{lemrud} the domain $\mathcal R_{II}$ is mapped by $G'(0)$ into $\mathcal R_{II}.$ Since $|\det G'(0)|=1,$ the mapping $G'(0)$ preserves the volume. Hence $G'(0)$ maps $\mathcal R_{II}$ onto $\mathcal R_{II},$ in particular it is a unitary operator. Compose $G$ with $(G'(0))^{-1}$ and then apply the Cartan theorem in order to find that $G$ is also unitary.

Let $\mathcal N=\{z\in\mathcal M_{2\times 2}(\mathbb C):\ z=z^t,\ ||z||\neq\rho(z)\},$ where $\rho$ denotes the spectral radius. Note that $\mathcal N\cap\mathcal R_{II}$ is open and dense in $\mathcal R_{II}.$ Moreover $\lambda z\in\mathcal N$ for any $z\in\mathcal N$ and $\lambda\in\mathbb C\setminus\{0\}.$ 
For $z\in\mathcal R_{II}$ define $D_z=\{\lambda z:\ ||\lambda z||<1\}\subset\mathcal R_{II}.$ 

Let $K$ be any compact subset of $\mathcal N\cap \mathcal R_{II}.$ Observe that 
\begin{equation}\label{deb}\bigcup\{D_z:\ z\in K\}\subset\varphi_{a}(\mathcal R_{II}\cap U)\end{equation} for $a\in\mathcal R_{II}$ sufficiently close to $1.$ Indeed, otherwise there would exist sequences $(\lambda_n)\subset\mathbb C,$ $(z_n)\subset K$ and $(a_n)\subset\mathcal R_{II}$ such that $a_n\to 1$ and $\lambda_n\to \lambda_0\in\mathbb C,$ $z_n\to z_0\in K$ and $\lambda_n z_n \not\in \varphi_{a_n}(\mathcal R_{II}\cap U)$ (pass to subsequences, if necessary). If $\lambda_0=0,$ then the contradiction is obvious. In the other case $\lambda_0 z_0\in\mathcal N\cap\overline{\mathcal R_{II}}.$ It follows that $\det(1-\lambda_0z_0)\neq 0$ (otherwise $\rho(\lambda_0 z_0)\geq1\geq||\lambda_0 z_0||$). This in particular means that $\varphi_{a_n}^{-1}(\lambda_n z_n)$ converges to $1$ (use the formula~(\ref{wzornaaut})). Whence $\lambda_n z_n\in \varphi_{a_n}(\mathcal R_{II}\cap U)$ for large $n$; a contradiction.

Since $G$ is unitary, $G^{-1}(\mathcal N)\cap\mathcal N$ is open and dense in $\mathcal R_{II}.$ Let $B=\{z\in\mathcal M_{2\times 2}(\mathbb C):\ z=z^t,\ ||z-p||<2c\},$ where $p=p^t$ and $c>0$ are chosen such that $$B\Subset G^{-1}(\mathcal N) \cap\mathcal N\quad \text{and}\quad ||z||<1-c\quad \text{for}\quad z\in B.$$ 

Property (\ref{deb}) yields the existence of an $n$ such that \begin{equation}\label{dgz}D_z\subset\varphi_{a_n}(U\cap\mathcal R_{II})\quad \text{and}\quad D_{G(z)}\subset\varphi_{b_n}(V\cap\mathcal R_{II}),\quad z\in B.\end{equation} We may assume that for such chosen $n:$ $||G_n(z)-G(z)||<c$ whenever $||z||\leq 1-c,$ $z=z^t.$

Then for $z=z^t$ such that $||z-p||<c,$ the set $D_z$ is contained in $\varphi_{a_n}(U\cap\mathcal R_{II}).$ Since $||G^{-1}(G_n(z))-p||=||G_n(z)-G(p)||<2c$ we see that $G^{-1}(G_n(z))\in B.$ Therefore, making use of (\ref{dgz}) we get that $D_{G_n(z)}\subset \varphi_{b_n}(V\cap\mathcal R_{II}).$

Thus we may use a standard argument to the mapping $G_n:D_z\to\mathcal R_{II},$ where $z=z^t$ is such that $||z-p||<c,$ in order to find that $||G_n(z)||\leq ||z||.$ The same argument applied to $G_n^{-1}:D_{G_n(z)}\to \mathcal R_{II}$ together with the previous inequality gives $||G_n(z)||=||z||$ for $||z-p||<c$, $z=z^t.$ Obviously this equality remains validate on the whole $\varphi_{a_n}(\mathcal R_{II} \cap U).$ 

Choose $r$ such that a ball $r\mathcal R_{II}$ is contained in $\varphi_{a_n}(\mathcal R_{II}\cap U)\cap\varphi_{b_n}(\mathcal R_{II}\cap V)$ for a large $n$. The restriction of $G_k$ to $r\mathcal R_{II}$ is an automorphism of $r\mathcal R_{II}$ fixing $0.$ So we conclude from the description of the group of automorphism of classical Cartan domain of the second type that $G_k$ is unitary. From this piece of information we immediately get the assertion.
\end{proof}

We are ready to show the correspondence between proper holomorphic self-mappings of the tetrablock and the Cartan domain of the second type.

\begin{lemma}\label{lemma2}
 Let $\varphi:\mathbb E\to\mathbb E$ be a proper holomorphic mapping. Then, there is $\psi\in\aut(\mathcal R_{II})$  such that \begin{equation}\varphi\circ\Pi=\Pi\circ\psi\end{equation}
\end{lemma} 

\begin{proof}[Proof of Lemma~\ref{lemma2}]
 First observe that $\Pi^{-1}({\mathbb E})=\mathcal R_{II},$ so it is very easy to see that $\Pi:\mathcal R_{II}\to\mathbb E$ is proper. Put $f:=\varphi\circ\Pi.$ By Corollary~\ref{pext} the mapping $f$ extends to an open neighborhood $\Omega_1$ of $\overline{\mathcal R_{II}}.$ Define \begin{equation}\mathcal J:=\{x\in\Omega_1:\ \det[f'(x)]\neq 0\ \text{and}\ f_1(x)f_2(x)\neq f_3(x)\}.\end{equation} Since every proper holomorphic mapping is non-degenerate, properties of the Shilov boundary show that the intersection of sets $\mathcal J$ and $\partial_s\mathcal R_{II}$ is non-empty. Take any $x_0\in\mathcal J\cap \partial_s\mathcal R_{II}.$

Fix a $y_0$ such that $\Pi(y_0) = f(x_0).$ The choice of $x_0$ and properties of covering
maps allow us to choose open neighborhoods $U,$ $V$ of $x_0,$ $y_0$ respectively and a
biholomorphic mapping such that \begin{equation}f=\Pi\circ \psi\quad \text{on}\ U.\end{equation} We find from Corollary~\ref{cor} that $f(x_0)$ lies in the Shilov boundary of the tetrablock. So $\psi(x_0)$ is unitary.

Lemma~\ref{bih} and the identity principle finish the proof.
\end{proof}
Now we are able to prove an Alexander-type theorem for the tetrablock.
\begin{proof}[Proof of Theorem~\ref{main}]
 It suffices to apply Lemma~\ref{lemma2} to get that the mapping $\varphi\circ\Pi$ has multiplicity $2.$ Since $\Pi$ also has multiplicity $2$ we infer that $\varphi$ is an automorphism.
\end{proof}

\begin{remark}\label{rem}Note that the tetrablock is not $\mathbb C$-convex.
Actually, let \begin{equation}\gamma(x)=|x_1-\overline{x_2}x_3|+|x_1x_2-x_3|+|x_3|^2,\quad \text{for}\quad
x=(x_1,x_2,x_3)\in\mathbb C^3.\end{equation} As shown in~\cite{You0}, $x\in\mathbb E$ if and only if
$\gamma(x)<1.$

For $\zeta\in\mathbb C$ put
$$\varphi(\zeta):=\left(\frac{1-i}{2}\zeta+\frac{1+i}{2},\frac{1+i}{2}\zeta+\frac{i-1}{2},i\zeta\right).$$
Obviously $\varphi(1),\ \varphi(-1)\in\overline{\mathbb E}.$ Moreover
$\varphi(i\zeta)=\left(\frac{1+i}{2}(\zeta+1),\frac{i-1}{2}(\zeta+1),-\zeta\right).$
An easy computation shows that for any $\zeta\in\mathbb R:$
\begin{align*}\gamma(\varphi(i\zeta))&=\left|\frac{1+i}{2}(\zeta+1)-\frac{i+1}{2}(\zeta+1)\zeta\right|+\left|-1/2(\zeta+1)^2+\zeta\right|+\zeta^2=\\
&=\frac{\sqrt{2}}{2}|1-\zeta^2|+|\frac{\zeta^2+1}{2}|+\zeta^2=\frac{\sqrt{2}}{2}|1-\zeta^2|+\frac{3}{2}\zeta^2+1/2.\end{align*} In
particular  $\gamma(\varphi(z))>1$ for any $z\in\{x\in\mathbb C:\ \rea x=0\},$ so $\mathbb E\cap\varphi(\mathbb C)$ is not connected.
\end{remark}

\end{document}